\documentclass[reqno, 11pt]{amsart}
\usepackage[margin=1in]{geometry}
\usepackage{amsmath,amssymb,amsthm,mathtools}
\usepackage{microtype}
\usepackage[hidelinks]{hyperref}
\usepackage{lineno}
\newtheorem{theorem}{Theorem}[section]
\newtheorem{lemma}[theorem]{Lemma}
\newtheorem{corollary}[theorem]{Corollary}
\newtheorem{proposition}[theorem]{Proposition}
\theoremstyle{remark}

\newcommand{\F}{\mathbb F}
\newcommand{\N}{\mathrm N}
\newcommand{\ex}{\operatorname{ex}}
\begin{document}
\title{Tripartite Zarankiewicz numbers and norm graphs}
\thanks{Research partially supported by NSF grant DMS-2300346.}
\author{Yantao Tang}
\address{Department of Mathematics and Statistics, Georgia State University, Atlanta, GA 30303}
\email{ytang26@gsu.edu}
\author{Yi Zhao}
\address{Department of Mathematics and Statistics, Georgia State University, Atlanta, GA 30303}
\email{yzhao6@gsu.edu}
\date{}
\begin{abstract}
For fixed integers $s\ge t\ge2$, let $\ex(n,n,n,K_{s,t})$ denote the maximum number of edges in a tripartite $K_{s,t}$-free graph with $n$ vertices in each part. When $s\ge(t-1)!+1$, let $r$ be the largest integer satisfying $s\ge(t-1)!r^{t-1}+1$.
Using the quotient norm graphs of Alon, R\'onyai and Szab\'o, we prove that
\[
\ex(n,n,n,K_{s,t}) \ge \left(\frac{3}{2^{1/t}}r^{1-1/t}+o(1)\right)n^{2-1/t}.
\]
Improving an upper bound of Tait and Timmons, we prove that, for all $s\ge t\ge 2$,
\[
 \ex(n,n,n,K_{s,t})\le \left(\frac{3}{2^{1/t}}(s-t+1)^{1/t}+o(1)\right)n^{2-1/t}.
\]
Together, these bounds recover the results of~\cite{LLF,Luo,TaitTimmons} for $t=2$, and give the new asymptotic formula
\[
 \ex(n,n,n,K_{3,3})
 =\left(\frac{3}{\sqrt[3]{2}}+o(1)\right)n^{5/3}.
\]
Analogous results extend to $k$-partite graphs containing no $K_{s, t}$ whose $s$-vertex or $t$-vertex side lies in a single part.
As an application of our tripartite construction, we determine the tripartite multicolor Ramsey number of $K_{3,3}$ asymptotically.
\end{abstract}

\maketitle

\section{Introduction}

Given graphs $G$ and $F$, we say that $G$ is \emph{$F$-free} if no subgraph of $G$ is isomorphic to $F$. The \emph{Tur\'an number of $F$ in $G$} is
\[
 \ex(G,F)=\max\{e(G') : G'\subseteq G\text{ and }G'\text{ is }F\text{-free}\}.
\]
In particular, $\ex(n,F)=\ex(K_n,F)$ is the classical Tur\'an number, while
\[
 \ex(n,n,F)=\ex(K_{n,n},F),\qquad
 \ex(n,n,n,F)=\ex(K_{n,n,n},F)
\]
are the \emph{(balanced) bipartite and tripartite Tur\'an numbers}, respectively.

The classical Zarankiewicz problem concerns $\ex(n,n,K_{s,t})$. We use the convention $s\ge t\ge2$ throughout the paper.
K\H{o}v\'ari, S\'os and Tur\'an~\cite{KST} proved
\begin{equation}\label{eq:kst-bound}
 \ex(n,n,K_{s,t})
 \le (s-1)^{1/t}n^{2-1/t}+(t-1)n.
\end{equation}
For integers $m\ge s$ and $n\ge t$, let $z(m,n,s,t)$ be the maximum number of edges in a bipartite graph with parts of sizes $m$ and $n$ containing no copy of
$K_{s,t}$ whose $s$-vertex side lies in the $m$-vertex part. F\"uredi~\cite{FurediUpper} sharpened the leading constant in \eqref{eq:kst-bound} as follows:
\begin{equation}\label{eq:furedi-bound}
 z(m,n,s,t)
 \le (s-t+1)^{1/t}nm^{1-1/t} +tm^{2-2/t}+tn.
\end{equation}
Nikiforov~\cite{Nikiforov} subsequently improved the lower-order terms.

Determining whether $\ex(n,K_{s,t})=\Omega(n^{2-1/t})$ for all fixed $s\ge t\ge2$ is a major open problem; see \cite{FurediSimonovitsSurvey} for a survey.
For $K_{2,2}=C_4$, Erd\H{o}s, R\'enyi and S\'os~\cite{ErdosRenyiSos} and,
independently, Brown~\cite{Brown} proved that $\ex(n,K_{2,2})=\left(\frac12+o(1)\right)n^{3/2}$.
More generally, F\"uredi~\cite{FurediAsymptotics} proved that, for every $s\ge2$, $\ex(n,K_{s,2})=\left(\frac{\sqrt{s-1}}2+o(1)\right)n^{3/2}$.
Brown~\cite{Brown} also constructed $K_{3,3}$-free graphs showing that
$\ex(n,K_{3,3})\ge\left(1/2-o(1)\right)n^{5/3}$.
Together with F\"uredi's bound \eqref{eq:furedi-bound} and the fact $2\ex(n, K_{s,t})\le \ex(n, n, K_{s,t})$, it follows that $\ex(n,K_{3,3})=\left(1/2+o(1)\right)n^{5/3}$.
Using norm graphs, Koll\'ar, R\'onyai and Szab\'o~\cite{KRS} proved that
$\ex(n,K_{s,t})=\Omega(n^{2-1/t})$ for $t\ge4$ and $s\ge t!+1$.
Alon, R\'onyai and Szab\'o~\cite{ARS} extended this result to $t\ge2$ and
$s\ge(t-1)!+1$ by modifying the construction of~\cite{KRS}.
In particular, the quotient variant of the constructions of~\cite{ARS} gives,
for every fixed $t\ge2$, $r\ge1$ and $s\ge(t-1)!r^{t-1}+1$,
\[
 \ex(n,K_{s,t})
 \ge\left(\frac12r^{1-1/t}-o(1)\right)n^{2-1/t}.
\]
More recently, Bukh~\cite{BukhExponential} proved that 
$\ex(n,K_{s,t})=\Omega(n^{2-1/t})$ for $s> 9^t t^{4 t^{2/3}}$, improving the factorial function on $t$ to an exponential one.

Tait and Timmons~\cite{TaitTimmons} initiated the systematic study of the \emph{multipartite Zarankiewicz problem}, that is, the problem of forbidding complete bipartite graphs in multipartite graphs. Multipartite Tur\'an problems had been studied earlier in the hypergraph setting, e.g., \cite{MubayiTalbot}, and more recently, e.g., \cite{HanZhao}.
For an integer $k\ge1$, let $\ex_{\chi\le k}(N,F)$ be the
maximum number of edges in an $N$-vertex $F$-free graph with chromatic number at most $k$.
Trivially, for any graph $F$, $\ex(n,\ldots,n,F)\le \ex_{\chi\le k}(kn,F)\le\ex(kn,F)$, where $n$ occurs $k$ times in the first term.
Under our convention $s\ge t$, Tait and Timmons~\cite{TaitTimmons} proved that
\begin{equation}\label{eq:tait-timmons-general}
 \ex_{\chi\le3}(N,K_{s,t})
 \le \left(
       \frac{(s-1)^{1/t}}{2^{1/t}3^{1-1/t}}+o(1)
     \right)N^{2-1/t}.
\end{equation}
They proved sharper estimates for $k\ge3$:
\begin{align}
 \ex_{\chi\le k}(N,K_{s,2})
 &\le \left(
       \frac{\sqrt{s-1}}2\left(1-\frac1k\right)^{1/2}+o(1)
      \right)N^{3/2},
 \label{eq:tait-timmons-k2s}\\
 \ex_{\chi\le k}(N,K_{3,3})
 &\le \left(
       \frac12\left(1-\frac1k\right)^{2/3}+o(1)
      \right)N^{5/3}.
 \label{eq:tait-timmons-k33}
\end{align}
For $k=3$, matching lower bounds for $K_{s,2}$ were given in~\cite{TaitTimmons} when $s$ is odd.
Lv, Lu and Fang~\cite{LLF} solved the case for $s=2$, and Luo~\cite{Luo} recently extended it to all even $s$. Together, these results imply that, for every $s\ge2$,
\begin{equation}\label{eq:known-k2t-tripartite}
  \ex_{\chi\le3}(N,K_{s,2}) =\left(\sqrt{\frac{s-1}{6}}+o(1)\right)N^{3/2},\quad
  \ex(n,n,n,K_{s,2}) =\left(3\sqrt{\frac{s-1}{2}}+o(1)\right)n^{3/2}.
\end{equation}

Our first theorem improves \eqref{eq:tait-timmons-general} by replacing $s-1$ with $s-t+1$, while also
recovering \eqref{eq:tait-timmons-k2s} and
\eqref{eq:tait-timmons-k33}.  The proof follows a standard approach: we apply F\"uredi's estimate \eqref{eq:furedi-bound} to the partitions determined by a proper coloring and then optimize the part sizes using concavity.

\begin{theorem}\label{thm:chromatic-upper}
Fix integers $k\ge2$ and $s\ge t\ge2$.  Then
\[
 \ex_{\chi\le k}(N,K_{s,t})
 \le
 \left(
 \frac{1}{2}(s-t+1)^{1/t}
 \left(1-\frac1k\right)^{1-1/t}
 +o(1)
 \right)N^{2-1/t}.
\]
\end{theorem}

Our next theorem gives a lower bound for $\ex(n,n,n,K_{s,t})$.
It follows from the more general Theorem~\ref{thm:onesided-lower} presented later.

\begin{theorem}\label{thm:tripartite-lower}
Let $s\ge t\ge2$ satisfy $s\ge(t-1)!+1$, and let $r$ be the largest integer
such that $s\ge(t-1)!r^{t-1}+1$.
Then
\[
 \ex(n,n,n,K_{s,t})
 \ge
 \left(\frac{3}{2^{1/t}}r^{1-1/t}-o(1)\right)n^{2-1/t}.
\]
\end{theorem}


When $t=2$, Theorems~\ref{thm:chromatic-upper} and
\ref{thm:tripartite-lower} together recover the formula
\eqref{eq:known-k2t-tripartite} established in \cite{LLF,Luo,TaitTimmons}.
We obtain the following corollary
by setting $t=3$ in Theorem~\ref{thm:tripartite-lower} and applying Theorem~\ref{thm:chromatic-upper} with $k=3$ and $N=3n$.

\begin{corollary}\label{cor:t3-bounds}
Fix an integer $s\ge3$, and let $r$ be the largest integer such that
$s\ge2r^2+1$. Then
\[
 \left(\frac{3}{\sqrt[3]{2}}r^{2/3}-o(1)\right)n^{5/3}
 \le \ex(n,n,n,K_{s,3})
 \le
 \left(\frac{3}{\sqrt[3]{2}}(s-2)^{1/3}+o(1)\right)n^{5/3}.
\]
In particular, when $s=3$, we have $r=1$ and hence,
\begin{equation}\label{eq:k33}
 \ex(n,n,n,K_{3,3})
 =\left(\frac{3}{\sqrt[3]{2}}+o(1)\right)n^{5/3}.
\end{equation}
\end{corollary}

A natural question is whether Theorem~\ref{thm:tripartite-lower} can be extended to the $k$-partite Zarankiewicz number $\ex(n, \dots, n, K_{s, t})$.
At present, we do not even know how to do so for $\ex(n,n,n,n,K_{2,2})$ --
the main difficulty is that a copy 
of $K_{2,2}$ may have four vertices from four different parts of the host graph.

This motivates a one-sided version of the multipartite problem.
We call a copy of $K_{s,t}$ in a multipartite graph \emph{one-sided} if its $s$-vertex or $t$-vertex side lies entirely in a single part.
Let $\ex_k^\ast(n,K_{s,t})$ denote the maximum number of edges in a $k$-partite graph with $n$ vertices in each part that contains no one-sided copy of $K_{s,t}$.
We obtain the following upper bound from 
F\"uredi's bound \eqref{eq:furedi-bound}. 

\begin{proposition}\label{thm:onesided-upper}
Fix integers $k\ge2$ and $s\ge t\ge2$. Then
\[
 \ex^*_{k}(n,K_{s,t})  \le
 \left( \frac{k}{2}(s-t+1)^{1/t} (k-1)^{1-1/t} +o(1) \right) n^{2-1/t}.
\]
\end{proposition}

If only one-sided copies are forbidden, then Theorem~\ref{thm:tripartite-lower} extends to an
arbitrary number of parts.

\begin{theorem}\label{thm:onesided-lower}
Fix integers $k\ge2$ and $s\ge t\ge2$ satisfying
$s\ge(t-1)!+1$, and let $r$ be the largest integer such that
$s\ge(t-1)!r^{t-1}+1$.
Then
\[
 \ex_k^\ast(n,K_{s,t})
 \ge
 \left(
  \frac{k}{2}(k-1)^{1-1/t}r^{1-1/t}-o(1)
 \right)n^{2-1/t}.
\]
\end{theorem}

The proof uses quotient norm graphs of Alon, R\'onyai and Szab\'o~\cite{ARS}. Let $H$ be the subgraph of the quotient norm graph after removing a small portion of vertices. 
We construct an automorphism $h$ of $H$ and a set $U\subseteq V(H)$ so that the \(k-1\) sets \(h^{2-k}U,h^{4-k}U,\ldots,h^{k-2}U\) partition $V(H)$. Using \(k\) copies of \(U\) as the parts, we define edges of the \(k\)-partite graph so that, for any set of vertices in one part, its common neighbors in the other \(k-1\) parts correspond to the common neighborhoods of these vertices in \(h^{2-k}U,h^{4-k}U,\ldots,h^{k-2}U\).

\emph{Every copy of $K_{s,t}$ in a tripartite graph is one-sided}: one part of the tripartite graph cannot meet both sides of $K_{s,t}$, and if neither side were contained in one part, then the copy of $K_{s,t}$ would meet at least four parts, which is impossible. Thus,
$\ex_3^\ast(n,K_{s,t})=\ex(n,n,n,K_{s,t})$, and
Theorem~\ref{thm:onesided-lower} with $k=3$ implies
Theorem~\ref{thm:tripartite-lower} because
$(3/2)2^{1-1/t}=3/2^{1/t}$.

Combining Proposition~\ref{thm:onesided-upper} and Theorem~\ref{thm:onesided-lower} gives the following corollary.
\begin{corollary}
For every $k\ge2$ and $s\ge 2$, we have
\begin{align*}
 \ex_k^\ast(n,K_{s,2})
 =\left(\frac{k}{2}\sqrt{(k-1)(s-1)}+o(1)\right)n^{3/2}, 
 \quad
 \ex_k^\ast(n,K_{3,3})
 =\left(\frac{k}{2}(k-1)^{2/3}+o(1)\right)n^{5/3}.
\end{align*}
\end{corollary}

Alon, R\'onyai and Szab\'o~\cite{ARS} applied norm graph constructions to the study of multicolor Ramsey numbers. Let $r_m(H)$ be the least $n$ such that every $m$-coloring of $K_n$ contains a monochromatic copy of $H$. They showed that
\begin{equation}\label{eq:rkK33}
 r_m(K_{3,3})=(1+o(1))m^3.
\end{equation}
Our construction behind Theorem~\ref{thm:onesided-lower} has a similar application.
For integers $m\ge1$ and $k\ge2$, let $r_m^{(k)}(H)$ be the least $n$ such that every $m$-coloring of the complete 
$k$-partite graph $K_{n, \dots, n}$ yields a monochromatic copy of $H$. This multipartite Ramsey parameter was introduced in~\cite{DayEtAl}. In the bipartite case, it is known that
$r_m^{(2)}(K_{2,s})=(s-1+o(1))m^2$ and $r_m^{(2)}(K_{3,3})=(1+o(1))m^3$ for every fixed $s\ge2$~\cite{WangLiLi}.
Our next theorem determines $r_m^{(3)}(K_{3,3})$ asymptotically.

\begin{theorem}\label{thm:ramsey}
We have $r_m^{(3)}(K_{3,3})=\left(\frac12+o(1)\right)m^3$.
\end{theorem}

\subsection*{Notation}
All graphs are simple except for norm graphs, which have loops.
Let $H$ be a graph that may have loops.
For $x\in V(H)$, let $\Gamma_H(x)$ denote the neighborhood of $x$ in $H$ (if there is a loop at $x$, then $x\in \Gamma_H(x)$).  
For $S\subseteq V(H)$, let $\Gamma_H(S)=\bigcap_{x\in S}\Gamma_H(x)$ denote the common neighborhood of $S$ in $H$.  We omit subscripts when the underlying graph is clear.  We write $A\sqcup B$ for the union of two disjoint sets $A$ and $B$.  

For a prime power $q$, we use $\F_q$ to denote the finite field with $q$
elements and $\F_q^\ast$ for its multiplicative group.  Let $\N(z)=\N_{\F_{q^{t-1}}/\F_q}(z)=z^{1+q+\cdots+q^{t-2}}$ be the norm of $z\in\F_{q^{t-1}}$ over $\F_q$.

\subsection*{Organization of the paper}

Section~\ref{sec:upper} proves Theorem~\ref{thm:chromatic-upper}
and Proposition~\ref{thm:onesided-upper}, while Section~\ref{sec:construction} gives norm graph constructions and proves
Theorem~\ref{thm:onesided-lower}.
The Ramsey application appears in Section~\ref{sec:applications}.

\section{Upper bounds}\label{sec:upper}

\begin{proof}[Proof of Theorem~\ref{thm:chromatic-upper}]
Let $G$ be an $N$-vertex $K_{s,t}$-free graph with a proper $k$-coloring of $V(G)=A_1\sqcup\cdots\sqcup A_k,$
where some $A_i$ may be empty.
For $1\le i\le k$, let $n_i=|A_i|$, $\delta_i={n_i}/{N}$, $B_i=V(G)\setminus A_i$ and $e_i=e(A_i,B_i)$.
Since $G$ is $K_{s,t}$-free, the bipartite subgraph $G[A_i,B_i]$ is also $K_{s,t}$-free for each $1\le i\le k$. If $n_i<t$ or $N-n_i<s$, then $e_i\le n_i(N-n_i)=O(N)$. For every other $i$, \eqref{eq:furedi-bound} gives
\[
 e_i\le z(N-n_i,n_i,s,t)\le (s-t+1)^{1/t}n_i(N-n_i)^{1-1/t}
 +t(N-n_i)^{2-2/t}+tn_i.
\]
Let $f(x)=x(1-x)^{1-1/t}$ for $x\in[0,1]$.
Every edge of $G$ is counted in exactly two of the $e_i$. Hence, summing the preceding bounds over $i$ gives
\begin{equation}\label{eq:upper-normalized}
 2e(G)=\sum_{i=1}^k e_i
 \le (s-t+1)^{1/t}N^{2-1/t}\sum_{i=1}^k f(\delta_i)+o(N^{2-1/t}).
\end{equation}

A direct calculation gives $f''(x)=\left(1-1/t\right)(1-x)^{-1-1/t}
 \left(\left(2-1/t\right)x-2\right)<0$ for $0<x<1$,
so $f$ is concave on $[0,1]$. Since $\sum_{i=1}^k\delta_i=1$, Jensen's inequality gives $\sum_{i=1}^k f(\delta_i)
 \le kf\left(\frac{1}{k}\right)
 =\left(1-\frac1k\right)^{1-1/t}$.
Substituting this inequality into \eqref{eq:upper-normalized} yields
\[
 e(G)\le
 \left(
  \frac{1}{2}(s-t+1)^{1/t}
  \left(1-\frac1k\right)^{1-1/t}+o(1)
 \right)N^{2-1/t}. \qedhere
\]
\end{proof}

\begin{proof}[Proof of Proposition~\ref{thm:onesided-upper}]
Let $G$ be a $k$-partite graph with parts $A_1,\ldots,A_k$, each of size $n$, and suppose that $G$ contains no
one-sided copy of $K_{s,t}$. 
For each $i$, $G[A_i,V(G)\setminus A_i]$ is $K_{s,t}$-free. Therefore, \eqref{eq:furedi-bound} gives
\begin{align*}
     e(A_i,V(G)\setminus A_i) &\le z((k-1)n,n,s,t)\\
     &\le (s-t+1)^{1/t}(k-1)^{1-1/t}n^{2-1/t}
     +t(k-1)^{2-2/t}n^{2-2/t}+tn.
\end{align*}
Summing over $i$ and noting that every edge of $G$ is counted twice, we obtain 
\[
 2e(G) \le k(s-t+1)^{1/t}(k-1)^{1-1/t}n^{2-1/t}
     +o(n^{2-1/t}),
\]
which gives the desired bound.
\end{proof}

\section{A multipartite construction from norm graphs}\label{sec:construction}
We first recall the construction of Alon, R\'onyai and Szab\'o~\cite[Section~4]{ARS}.
Given integers $s\ge t\ge2$ with $s\ge(t-1)!+1$, let $r$ be the largest integer
such that $s\ge(t-1)!r^{t-1}+1$.
Let $q$ be a prime power such that $r$ divides $q-1$, and let $Q_r$ be the subgroup of
$\F_q^\ast$ of order $r$.  Let $X=\F_{q^{t-1}}\times(\F_q^\ast/Q_r)$.
Note that $|X|=q^{t-1}(q-1)/r=(q^t-q^{t-1})/r$.
Construct a graph $H_r=H_r(q,t)$ with vertex set $X$ by declaring two (not necessarily distinct) vertices $x=(A,aQ_r)$ and $y=(B,bQ_r)$ adjacent if
\begin{equation}\label{eq:norm-relation}
 x\sim y \quad\Longleftrightarrow\quad
 \N(A+B)\in abQ_r.
\end{equation}
(Note that $x\sim x$ if and only if $\N(2A)\in a^2Q_r$.)
For a vertex $x=(A,aQ_r)$, relation~\eqref{eq:norm-relation} determines a unique
coset $bQ_r$ for each $B\ne-A$. Hence $|\Gamma(x)|=q^{t-1}-1$, including possible loops.
Alon, R\'onyai and Szab\'o~\cite[Section~4]{ARS} showed that\footnote{Note that \eqref{eq:Kst-free} is stronger than $H_r$ being $K_{(t-1)!r^{t-1}+1,t}$-free because \eqref{eq:Kst-free} means that there are no sets $S, T\subseteq X$, not necessarily disjoint, with $|S|=(t-1)!r^{t-1}+1$, $|T|=t$ such that $x\sim y$ for all $x\in S$ and $y\in T$.} 
\begin{align}
\label{eq:Kst-free}
|\Gamma_{H_r}(T)|\le (t-1)!r^{t-1}< s \quad \text{for all} \ T\subseteq X \text{ with } |T|=t.
\end{align}

Our construction begins by deleting the vertices whose first coordinate is zero. Let $X^\ast=\F_{q^{t-1}}^\ast\times(\F_q^\ast/Q_r)$. Every vertex of $X^\ast$ loses exactly one neighbor, namely the unique vertex with first coordinate zero. Therefore, every vertex of $H_r[X^\ast]$ has exactly $q^{t-1}-2$ neighbors.

Fix $k\ge2$ and suppose that $q\equiv1\pmod{4r(k-1)}$.
Let $H=H_r[X^\ast]$.
Choose $\omega\in\F_q^\ast$ of order $4(k-1)$, set $\iota=\omega^2$, and
let $L=1+q+\cdots+q^{t-2}$. We define a bijection $h: X^\ast \to X^\ast$ by $h(A,aQ_r)=(\iota A,\omega^L aQ_r)$. Consider $x=(A,aQ_r)$ and $y=(B,bQ_r)$ in $X^*$. 
Since $\iota\in\F_q^\ast$,
\[
 \N(\iota(A+B))=\iota^L\N(A+B)
 \quad\text{and}\quad
 (\omega^L a)(\omega^L b)=\iota^L ab.
\]
Consequently,
\[
 h(x)\sim h(y)
 \Longleftrightarrow \N(\iota(A+B))\in(\omega^L a)(\omega^L b)Q_r
 \Longleftrightarrow \iota^L\N(A+B)\in\iota^L abQ_r.
\]
It follows that $ x\sim y$ if and only if $h(x)\sim h(y)$. Thus $h$ is an automorphism of $H$.

Set $g=h^2$. For $(A,aQ_r)\in X^\ast$, it follows that $g(A,aQ_r)=(\omega^4A,\omega^{2L}aQ_r)$. The element $\omega^4$ has order $k-1$, and $g^{2(k-1)}$ is the identity in
$X^\ast$. If $g^d(A,aQ_r)=(A,aQ_r)$, then $\omega^{4d}A=A$; since $A\ne0$, it follows that $(k-1)\mid d$. Thus every cycle of $g$ has length $k-1$ or
$2(k-1)$. By the disjoint-cycle decomposition of a finite permutation, these cycles partition $X^\ast$. For each cycle $C$, choose $x_C\in C$ and define a set $U\subseteq X^*$ by
\[
 U\cap C=\{g^{j(k-1)}(x_C):0\le j<|C|/(k-1)\}.
\]
For each $C$, the sets $g^a(U\cap C)$, $0\le a\le k-2$, partition $C$.
Consequently,
\[
 X^\ast=\bigsqcup_{a=0}^{k-2}h^{2a}U, \quad \text{and}
 \quad
 |U|=\frac{|X^\ast|}{k-1}
 =\frac{(q^{t-1}-1)(q-1)}{r(k-1)}=: n_q.
\]

Now we construct a $k$-partite graph $G_q^{(k)}$. Index its parts by $\mathbb Z/k\mathbb Z$, and let $U_i=U\times\{i\}$ be the parts. We write $x_i\in U_i$ for $(x,i)\in U_i$. 
For distinct $i,j\in\mathbb Z/k\mathbb Z$, let
$d(i,j)= j - i \pmod{k} \in\{1,\ldots,k-1\}$ and
set $\ell(i,j)=k-2d(i,j)$. Since $d(j,i)=k-d(i,j)$, we have $\ell(j,i)=-\ell(i,j)$. 
For $x,y\in U$, join $x_i\in U_i$ and $y_j\in U_j$ if and only if $x\sim h^{\ell(i,j)}(y)$ in $H$. Since $h$ is an automorphism, 
\[
 x\sim h^{\ell(i,j)}(y)
 \quad\Longleftrightarrow\quad
 y\sim h^{-\ell(i,j)}(x)
 \quad\Longleftrightarrow\quad
 y\sim h^{\ell(j,i)}(x).
\]
The adjacency of $G_q^{(k)}$ is thus symmetric.

For a fixed $i$, as $j$ ranges over $\mathbb Z/k\mathbb Z\setminus\{i\}$,
the value $d(i,j)$ ranges over $1,\ldots,k-1$.  Hence
\[
 \{\ell(i,j):j\ne i\}
 =\{k-2,k-4,\ldots,2-k\}
 =\{2-k+2a:0\le a\le k-2\}.
\]
Applying $h^{2-k}$ to the earlier partition of $X^\ast$ gives
\begin{align}\label{eq:X*} 
X^\ast=\bigsqcup_{j\ne i}h^{\ell(i,j)}U.
\end{align}
For $T\subseteq U$, write $T_i=\{x_i:x\in T\}$. Recall that $\Gamma_H(T)$ denotes the common neighborhood of $T$ in $H$.

\begin{lemma}\label{lem:multipartite-slicing}
For every nonempty set $T\subseteq U$ and $i\in\mathbb Z/k\mathbb Z$, we have $|\Gamma_{G_q^{(k)}}(T_i)|=|\Gamma_H(T)|$.
\end{lemma}

\begin{proof}
Fix $j\ne i$.
For $y\in U$, we have
\[
\begin{aligned}
 y_j\in\Gamma_{G_q^{(k)}}(T_i)
 &\Longleftrightarrow x\sim h^{\ell(i,j)}(y)\text{ for every }x\in T\\
 &\Longleftrightarrow h^{\ell(i,j)}(y)\in\Gamma_H(T)\cap h^{\ell(i,j)}U.
\end{aligned}
\]
For every $j\ne i$, the map $y_j\longmapsto h^{\ell(i,j)}(y)$ is a bijection from
$\Gamma_{G_q^{(k)}}(T_i)\cap U_j$ to $\Gamma_H(T)\cap h^{\ell(i,j)}U$. Using \eqref{eq:X*}, we obtain that
\[
|\Gamma_{G_q^{(k)}}(T_i)|=\sum_{j\ne i} |\Gamma_{G_q^{(k)}}(T_i)\cap U_j|
 =\sum_{j\ne i} |\Gamma_H(T)\cap h^{\ell(i,j)}U| =|\Gamma_H(T)|. \qedhere
\]
\end{proof}

\begin{proof}[Proof of Theorem~\ref{thm:onesided-lower}]
Given a prime power $q\equiv1\pmod{4r(k-1)}$, we construct the graph $G_q^{(k)}$ as above.
Suppose that $G_q^{(k)}$ contains a one-sided copy of $K_{s,t}$. Let $T_i$ be the side of $K_{s,t}$ contained in $U_i$ and let $T=\{x\in U: x_i\in T_i\}$. 
If $|T|=t$, then Lemma~\ref{lem:multipartite-slicing} gives $|\Gamma_{H}(T)|= |\Gamma_{G_q^{(k)}}(T_i)|\ge s$, contradicting \eqref{eq:Kst-free}.
If $|T|=s$, then Lemma~\ref{lem:multipartite-slicing} gives $|\Gamma_{H}(T)|= |\Gamma_{G_q^{(k)}}(T_i)|\ge t$. Choose a set $S\subseteq \Gamma_{H}(T)$ of size $t$. Then $|\Gamma_H(S)|\ge |T| = s$, again contradicting \eqref{eq:Kst-free}.
Thus, $G_q^{(k)}$ contains no one-sided copy of $K_{s,t}$.

Taking $T=\{x\}$ in Lemma~\ref{lem:multipartite-slicing} shows that $G_q^{(k)}$ is $(q^{t-1}-2)$-regular. Thus, $e(G_q^{(k)}) = \frac{k}{2}n_q(q^{t-1}-2)$.
Since $n_q=(1+o(1))q^t/(r(k-1))$, we have $q^{t-1}=(1+o(1))\bigl(r(k-1)n_q\bigr)^{1-1/t}$. Consequently,
\[
 e(G_q^{(k)}) = \frac{k}{2}n_q\left( r(k-1)n_q\right)^{1-1/t} (1+o(1)) =
 \left(
    \frac{k}{2}(k-1)^{1-1/t}r^{1-1/t}+o(1)
  \right)n_q^{2-1/t},
\]
which gives the desired bound when $n=n_q$.

For sufficiently large $n$, set $x=(r(k-1)n)^{1/t}$. By the prime number theorem in arithmetic progressions (see, e.g. \cite{Davenport}), there is a prime $q\equiv1\pmod{4r(k-1)}$ such that $q\le x$ and $q=(1-o(1))x$. Since $n_q<q^t/(r(k-1))\le n$, we may add isolated vertices to each part of $G_q^{(k)}$ to obtain a graph with $n$ vertices in each part. Moreover, $q=(1-o(1))x$ implies $n_q=(1-o(1))n$, so the resulting graph has the asserted number of edges.
\end{proof}

\section{A tripartite Ramsey application}\label{sec:applications}
We prove Theorem~\ref{thm:ramsey} by following the approach used in the proof of \cite[Theorem 3]{ARS}.

\begin{proof}[Proof of Theorem~\ref{thm:ramsey}]
For the upper bound, suppose that the edges of $K_{n,n,n}$ are colored with $m$ colors and there is no monochromatic $K_{3,3}$. Let $G_c$ be the spanning subgraph consisting of the edges of color $c$. Then \eqref{eq:k33} in Corollary~\ref{cor:t3-bounds} gives
\[
 3n^2=\sum_{c=1}^m e(G_c)
 \le m\left(\frac3{\sqrt[3]{2}}+o(1)\right)n^{5/3}.
\]
This implies that $n\le(1/2+o(1))m^3$, which gives the upper bound for $r_m^{(3)}(K_{3,3})$.

For the lower bound, we use a variant of the graph introduced at the beginning of Section~\ref{sec:construction}. Suppose that $s=t=3$, so $r=1$. Let $q\equiv1\pmod{8}$ be a prime power and $X=\F_{q^2}\times\F_q^\ast$.
For each $\lambda\in\F_q^\ast$, define a graph $H_\lambda$ on $X$ by the symmetric relation
\[
 (A,a)\sim_\lambda(B,b)
 \quad\Longleftrightarrow\quad
 \N(A+B)=\lambda ab.
\]
The graph $H_\lambda$ has a loop at $(A, a)$ exactly when $N(2A)= \lambda a^2$.
Alon, R\'onyai and Szab\'o~\cite[Section~3]{ARS} noted that 
\begin{align}
\label{eq:K33-free}
|\Gamma_{H_\lambda}(T)|\le 2 \quad \text{for all} \ T\subseteq X \text{ with } |T|=3.
\end{align}

Choose $\omega\in\F_q^\ast$ of order eight and set $\iota=\omega^2$. 
Let $X^\ast:=\F_{q^2}^\ast\times\F_q^\ast$ and define $h:X^\ast \to X^\ast$ by $h(A,a)=(\iota A,\omega^{1+q}a)$. Then $h$ is an automorphism of $H_\lambda[X^\ast]$. Apply the construction of Section~\ref{sec:construction} with three parts, that is, choose $U\subseteq X^\ast$ such that $X^\ast=h^{-1}U\sqcup hU$, and define the tripartite graph $G_q(\lambda)$ on $U_1\sqcup U_2\sqcup U_3$ by joining $x=(A,a)\in U_i$ and $y=(B,b)\in U_{i+1}$ if and only if
\[
x\sim_\lambda h(y) \quad\Longleftrightarrow\quad
 \N(A+\iota B)=\lambda \omega^{1+q} ab.
\]
Lemma~\ref{lem:multipartite-slicing} applies with respect to the relation $\sim_\lambda$. Using \eqref{eq:K33-free}, the proof of Theorem~\ref{thm:onesided-lower} shows that $G_q(\lambda)$ contains no one-sided copy of $K_{3,3}$. Hence, $G_q(\lambda)$ is $K_{3,3}$-free.

For each $i\in\{1,2,3\}$ and every $x=(A,a)\in U_i$ and $y=(B,b)\in U_{i+1}$ satisfying $A+\iota B\ne0$, assign to $xy$ the color
\[
 \chi(xy)=\frac{\N(A+\iota B)}{\omega^{1+q}ab}\in\F_q^\ast.
\]
In other words, an edge $xy$ has color $\lambda$ precisely when $x\sim_\lambda h(y)$. Thus none of the colors in $\F_q^\ast$ contains a copy of $K_{3,3}$.

For each $i\in\{1,2,3\}$ and $A\in\F_{q^2}^\ast$, let $Q_{i,A}$ be the complete bipartite graph whose two sides are
\[
 \{(A,a)_i:(A,a)\in U\}
 \quad\text{and}\quad
 \{(\iota A,b)_{i+1}:(\iota A,b)\in U\}.
\]
These are precisely the uncolored pairs between $U_i$ and $U_{i+1}$ because $A+\iota(\iota A)=A+\iota^2A=0$. For each fixed $i$, the graphs $Q_{i,A}$ are vertex-disjoint as $A$ varies, and each of their sides has size at most $q-1$.

Set $\rho=\lceil2q^{1/3}\rceil$. By \eqref{eq:rkK33}, we have $r_{\rho}(K_{3,3}) = (1+o(1))\rho^3 >2(q-1)$ for all sufficiently large $q$. Hence there is a $\rho$-coloring of the complete graph on $2(q-1)$ vertices with no monochromatic $K_{3,3}$. 
For each $i$, color all $Q_{i,A}$ by such a coloring. 
Since the graphs $Q_{i,A}$, $A\in\F_{q^2}^\ast$, are vertex-disjoint, 
no monochromatic $K_{3,3}$ is created. Use $3\rho$ new colors for the three pairs of host parts, disjoint from $\F_q^\ast$. Altogether, we use $(q-1)+3\rho=q+o(q)$ colors on parts of size
\[
 n_q=\frac{(q^2-1)(q-1)}2
 =\left(\frac12+o(1)\right)q^3.
\]
Suppose $m$ is sufficiently large.
The prime number theorem in arithmetic progressions (see~\cite{Davenport}) allows us to choose a
prime $q\equiv1\pmod8$ such that $q=(1-o(1))m$ and $q+3\rho\le m$.
The coloring above uses at most $m$ colors and gives $r_m^{(3)}(K_{3,3})  >n_q=\left(\frac12-o(1)\right)m^3$.
\end{proof}

\section*{Acknowledgment}
The authors thank Allan Lo, Aram Mathivanan, and Simona Boyadzhiyska for valuable discussions during the early stages of this research.

\subsection*{Declaration on the use of AI}
ChatGPT was used to assist with the preparation and presentation of this manuscript. The authors independently verified all arguments and references and take full responsibility for the manuscript.

\end{document}